\documentclass[a4paper,12pt]{article}

\usepackage{latexsym, amssymb, amsfonts,amsmath}
\usepackage{graphicx}
\usepackage{relsize}

\newtheorem{definition}{Definition}[section]
\newtheorem{theorem}[definition]{Theorem}
\newtheorem{lemma}[definition]{Lemma}
\newtheorem{proposition}[definition]{Proposition}
\newtheorem{remark}[definition]{Remark}

\newtheorem{example}[definition]{Example}

\newtheorem{corollary}[definition]{Corollary}
\newtheorem{problem}{Problem}

\newcommand{\sfrac}[2]{{\sp{{#1}\!}/_{\!{#2}}}}

\newenvironment{proof}{\noindent\textbf{Proof}. \rm}{\hfill $\square$ \medskip}
\newenvironment{dedication}{\hspace*{\fill}\em }{}

\begin{document}

\title{Some Model Theory of Hypergeometric and Pfaffian Functions}

\author{Ricardo Bianconi}

\maketitle

\begin{dedication}
Dedicated to Francisco Miraglia on his 70\textsuperscript{th} birthday.
\end{dedication}

\begin{abstract}
	We present some results and open problems related to expansions of the field of real numbers by hypergeometric and related functions focussing on definability and model completeness questions. In particular, we prove the strong model completeness for expansions of the field of real numbers by the exponential, arctangent and hypergeometric functions. We pay special attention to the expansion of the real field by the real and imaginary parts of the hypergeometric function ${_2F_1}(\sfrac{1}{2},\sfrac{1}{2};1;z)$ because of its close relation to modular functions.
\end{abstract}


\tableofcontents

\section*{Introduction}

This work is an extended version of the talk given in the workshop ``{\em Logic and Applications: in honour to Francisco Miraglia by the occasion of his 70th birthday}'', September, 16-17, 2016, at the University of S\~ao Paulo, SP, Brazil.

We deal with a research project related to the model theory of the field of real numbers enriched with real analytic functions, resulting in an o-minimal structure.

O-minimality is a branch of Model Theory which has been very useful recently in proofs of Andr\'e-Oort conjecture (an important problem in Algebraic Geometry) by Jonathan Pila and others, see \cite{pila2011,klinger-ullmo-yafaev2014,gao2015}. One of the main ingredients is a diophantine counting result due to J. Pila and Alex Wilkie, \cite{pila-wilkie2006}, where it is stated the {\em Wilkie's Conjecture}, a sharper bound on such counting which is not always true, but it holds in some particular cases, see work by Binyamini and Novikov in \cite{binyamini-novikov2016}. This is discussed in Section \ref{wilkie-conjecture}.

Other direction of this project, not unrelated with the above, deals with decidability problems for some of expansions of the field of real numbers by some analytic functions. The first result in this direction is Macintyre's and Wilkie's proof of the decidability of the first order theory of the real exponential field in 1996, \cite{macintyre-wilkie1996}, and extended by Macintyre in the 2000's to elliptic integrals, \cite{macintyre2003,macintyre2008}. These works rely on the assumption of transcendental number theoretic conjectures, which seem to be out of reach of the present methods.

One common thread linking the two paths of research broached in the previous paragraphs is the theory of {\em Pfaffian functions}. The results of Macintyre and Wilkie, \cite{macintyre-wilkie1996,macintyre2003,macintyre2008}, rely on a decidable version of Wilkie's ground breaking proof of the model completeness af the expansions of the real field by (restricted) Pfaffian functions and also by the (unrestricted) exponential function, \cite{wilkie1996}. The work of Binyamini and Novikov, \cite{binyamini-novikov2016}, contains a discussion of the possibility of proving Wilkie's conjecture for expansions of the real field by Pfaffian functions, with the intention of finding computable bounds to the counting arguments of that conjecture.

So one of the main focus of this work is the theme of Pfaffian functions discussed in Section \ref{pfaffian-functions}, where we survey some of its theory.  In this section we present the larger class of Noetherian functions and state the first open problem, relating them to the Pfaffian functions.
An important property of the theories studied here is o-minimality, so in order to guarantee that it holds in the structures we prove a model completeness test in Section \ref{model-completeness-test}. Section \ref{complex-differential-equations} contains some partial results in the direction of the first open problem and we state particular cases of this open problem related to first order linear differential equations and comments on the difficulties when we treat second order equations. In this section we see that there appears a non linear first order equation (Riccati's equation) that becomes a system of two equations we still have no way to transform into the Pfaffian setting. In Section \ref{hypergeometric-functions} we treat hypergeometric functions and its relation to modular functions and prove a model completeness result with the methods of Section \ref{model-completeness-test}. This would give a decidable version of the author's work on the model completeness of expansions of the real field by such functions, \cite{bianconi2015}. We present a short discussion of Wilkie's conjecture in Section \ref{wilkie-conjecture} and end this work with some final remarks, Section \ref{final-remarks}.

\medskip

Scattered in the text there are seven \textsc{Open Problems}: (\ref{open-problem-1}): p. \pageref{open-problem-1}; (\ref{model-complete-pfaffian}): p. \pageref{model-complete-pfaffian}; (\ref{Riccati-Pfaffian}): p. \pageref{Riccati-Pfaffian}; (\ref{pfaffian-hypergeometric-model-complete}): p. \pageref{pfaffian-hypergeometric-model-complete}; (\ref{exp-definable-2F1-real}): p. \pageref{exp-definable-2F1-real}; (\ref{wilkie-conjecture-pfaffian}): p. \pageref{wilkie-conjecture-pfaffian}; (\ref{wilkie-conjecture-modular}): p. \pageref{wilkie-conjecture-modular}.

\medskip

\textsc{Notation:} $\mathbb{R}$ denotes the field of real numbers; $\mathbb{C}$, the field of complex numbers;  $\Re(z)$, the real part of the complex number, and $\Im(z)$, its imaginary part; $\mathbb{H}$, the upper half plane $\{z\in\mathbb{C}:\Im(z)>0\}$; $\bar{z}$ the complex conjugate of $z$; $D_{\rho}(z_0)$ denotes the open disk $\{z\in\mathbb{C}:|z-z_0|<\rho\}$; $f\circ g$ indicates the composition of functions $f$ and $g$, $f\circ g(x)=f(g(x))$.

\section{A Model Completeness Test}\label{model-completeness-test}

In previous model completeness results, \cite{bianconi2015,bianconi2016}, the author has made use of the following tests. These tests imply a strong form of model completeness.

\begin{definition}
We say that s set $X\subseteq\mathbb{R}^n$ is \emph{strongly definable} in the structure $\mathcal{R}$ if it is definable by an existential formula $\exists \bar{y}\,\varphi(\bar{x},\bar{y})$ such that for all $\bar{a}\in X$, there is a unique $\bar{b}$ such that $\mathcal{R}\models\varphi(\bar{a},\bar{b})$. The (first order theory of the) structure $\mathcal{R}$ is \emph{strongly model complete} if every definable set is strongly definable.
\end{definition}

\begin{theorem}[{\cite[Theorem 2]{bianconi2015}}]\label{model-completeness-W-system} 
Let $\hat{R}=\langle\mathbb{R},+,-$, $\cdot,<$, $(F_{\lambda})_{\lambda\in\Lambda}$, con\-stants$\rangle$ be an expansion of the field of real numbers, where for each $\lambda\in\Lambda$, $F_{\lambda}$ is the restriction to a compact poly-interval $D_{\lambda}\subseteq\mathbb{R}\sp{n_{\lambda}}$ of a real analytic function whose domain contains $D_{\lambda}$, and defined as zero outside $D_{\lambda}$, such that there exists a complex analytic function $g_{\lambda}$ defined in a neighbourhood of a poly-disk $\Delta_{\lambda}\supseteq D_{\lambda}$ and such that
\begin{enumerate}
\item
$g_{\lambda}$ is strongly definable in $\hat{R}$ and the restriction of $g_{\lambda}$ to $D_{\lambda}$ coincides with $F_{\lambda}$ restricted to the same set;
\item
for each $\boldsymbol{a}\in\Delta_{\lambda}$ there exists a compact poly-disk $\Delta$ centred at $\boldsymbol{a}$ and contained in the domain of $g_{\lambda}$, such that all the partial derivatives of the restriction of $g_{\lambda}$ to $\Delta$ are strongly definable in $\hat{R}$.
\end{enumerate}
Under these hypotheses, the theory of $\hat{R}$ is strongly model complete.
\end{theorem}

Now we introduce the unrestricted exponential function.

\begin{theorem} [{\cite[Theorem 4]{bianconi2015}}] \label{model-completeness-W-system-exp} 
Let $\hat{R}$ be the structure described in Theorem \ref{model-completeness-W-system}. We assume that the functions
$$
\exp\lceil_{[0,1]}(x)=
\left\{
\begin{array}{lcl}
\exp x & \mbox{if} & 0\leq x\leq 1,\\
0 & & \mbox{otherwise};
\end{array}
\right.
$$
$$
\sin\lceil_{[0,\pi]}(x)=
\left\{
\begin{array}{lcl}
\sin x & \mbox{if} & 0\leq x\leq \pi,\\
0 & & \mbox{otherwise},
\end{array}
\right.
$$
have representing function symbols in its language.
The expansion $\hat{R}_{\mathrm{exp}}$ of $\hat{R}$ by the inclusion of the (unrestricted) exponential function ``$\exp$'' is strongly model complete.
\end{theorem}

We present here a simplification appropriate to the envisaged applications.

\begin{proposition}\label{real-imaginary}
Suppose that the real and imaginary parts, $F_R(x,y)$ and $F_I(x,y)$, of the complex analytic function $F(z)$, $z=x+iy$, defined in a poly-disk $\Delta_{\rho}=\{z\in\mathbb{C}\sp{n}:|z_i|<\rho,~1\leq i\leq n\}$, are strongly definable in the structure $\mathcal{R}$, which expands the field of real numbers. Then are real analytic functions and admit complex analytic continuations
 $$
\tilde{F}_R(z,w)=\frac{F(z+iw)+\overline{F(\bar{z}+i\bar{w})}}{2},~~
\tilde{F}_I(z,w)=(-i)\frac{F(z+iw)+\overline{F(\bar{z}+i\bar{w})}}{2},
$$
with $z,w\in\Delta_{\rho/2}$, which are strongly definable in $\mathcal{R}$.
\end{proposition}

\begin{proof}
This is a consequence of the equalities $\Im(F)=\Re(-iF)$ and
$$
\Re(F)(x+iy)=\frac{F(x+iy)+\overline{F(\overline{x+iy})}}{2}.
$$

If $F(x+iy)=\sum_{\alpha\in\mathbb{N}\sp{n}}c\sb{\alpha}(x+iy)\sp{\alpha}$, $\overline{F(\overline{x+iy})}=\sum_{\alpha\in\mathbb{N}\sp{n}}\bar{c}\sb{\alpha}(\bar{x}+i\bar{y})\sp{\alpha}$. If $z,w\in\Delta\sb{\rho/2}$, then $z+iw,\bar{z}+i\bar{w}\in\Delta\sb{\rho}$.
\end{proof}


The following result is an immediate consequence of this proposition applied to Theorem \ref{model-completeness-W-system}.

\begin{theorem}\label{theorem-model-completeness-test}
Let $\hat{R}=\langle\mathbb{R},\mathit{constants},+,-,\cdot,<,(F_{\lambda})_{\lambda\in\Lambda}\rangle$ be an expansion of the field of real numbers, where for each $\lambda\in\Lambda$, $F_{\lambda}$ is the restriction to a compact poly-interval $D_{\lambda}\subseteq\mathbb{R}\sp{n_{\lambda}}$ of a real analytic function whose domain contains $D_{\lambda}$, and defined as zero outside $D_{\lambda}$, such that for each $\lambda\in \Lambda$ the function $F_{\lambda}$ and all its partial derivatives of all orders admit strongly definable real analytic extension to poly-intervals twice as big as $D_{\lambda}$, and zero outside, then the theory of $\hat{R}$ is strongly model complete.\hspace*{\fill}$\square$
\end{theorem}

\section{Pfaffian Functions}\label{pfaffian-functions}

Pfaffian functions were introduced by Askold Khovanskii in 1980 in his seminal paper {\em On a Class of Systems of Transcendental Equations}, \cite{khovanskii1980}. In this paper he proved that there exists a computable bound to the number of non singular zeros of a system of equations with Pfaffian functions, and as a consequence, a computable bound to the sum of the Betti numbers of the set of zeros of a system of such equations (see an extended exposition of these results in \cite{marker1997}). This plays an important role in the proof of the model completeness of expansions of the field of real numbers with Pfaffian functions by Alex Wilkie, \cite{wilkie1996}, and its decidable version by Angus Macintyre and Alex Wilkie, \cite{macintyre-wilkie1996}. Noetherian functions are closely related to Pfaffian functions and were introduced by Jean-Claude Tougeron in 1991, \cite{tougeron1991}. They may not be Pfaffian functions (for instance, $\sin x$ is Noetherian and non Pfaffian) but there has been some research about local finiteness results, see \cite{gabrielov-vorobjov2004}.

We start defining Pfaffian and Noetherian functions.

\begin{definition}[Pfaffian Functions] \rm
Let $U\subseteq\mathbb{R}\sp{n}$ be an open set. A finite sequence of smooth functions $f_j:U\to\mathbb{R}$, $1\leq i\leq N$, is a \emph{Pfaffian chain} if there exist real polynomials $P_{i,j}(\bar{x},y_1,\dots,y_i)$, $1\leq i\leq N$ and $1\leq j\leq n$, such that
$$
\frac{\partial f_i}{\partial x_j}=P_{i,j}(\bar{x},f_1,\dots,f_i),
$$
or, equivalently,
$$
df_i(\bar x)=\sum_{j=1}\sp{n}P_{i,j}(\bar{x},f_1,\dots,f_i)\,dx_j.
$$
A \emph{Pfaffian function} is any function which belongs to a Pfaffian chain.
\end{definition}

\begin{definition}[Noetherian Functions] \rm
Let $U\subseteq\mathbb{R}\sp{n}$ be an open set. A finite sequence of smooth functions $f_j:U\to\mathbb{R}$, $1\leq i\leq N$, is a \emph{Noetherian chain} if there exist real polynomials $P_{i,j}(\bar{x},y_1,\dots,y_N)$, $\bar{x}=(x_1,\dots,x_n)$, $1\leq i\leq N$ and $1\leq j\leq n$, such that
$$
\frac{\partial f_i}{\partial x_j}=P_{i,j}(\bar{x},f_1,\dots,f_N),
$$
or, equivalently,
$$
df_i(\bar x)=\sum_{j=1}\sp{n}P_{i,j}(\bar{x},f_1,\dots,f_N)\,dx_j.
$$
A \textbf{Noetherian function} is any function belonging to a Noetherian chain.
\end{definition}

\begin{remark} \rm
Pfaffian and Noetherian functions are real analytic functions.
\end{remark}

Some basic properties of these functions are stated in the following lemma, whose proof is straightforward.

\begin{lemma}\label{lema-basic-properties}
If $f,g:U\subseteq\mathbb{R}\sp{n}\to\mathbb{R}$ are Pfaffian (respectively Noetherian) functions then $f+g$, $f\cdot g$ and $1/f$ (if $f(\bar x)\neq 0$) are Pfaffian (respectively Noetherian) functions. If $f:U\subseteq\mathbb{R}\sp{n}\to\mathbb{R}$ and $g_1,\dots,g_n:V\subseteq\mathbb{R}\sp{m}\to\mathbb{R}$ are Pfaffian (respectively Noetherian) functions, such that for all $\bar{x}\in V$, $(g_1(\bar x), \dots,g_n(\bar x))\in U$, then $h:\bar x\in V\to h(\bar x)=f(g_1(\bar x), \dots,g_n(\bar x))\in \mathbb{R}$ is a Pfaffian (respectively Noetherian) function. If all the partial derivatives $\partial f/\partial x_i$ of the function $f$ are Pfaffian (respectively Noetherian) functions, then $f$ is a Pfaffian (respectively Noetherian) function. 
\end{lemma}

We want to reduce the logic questions (such as model completeness, decidability) about Noetherian functions to the case of Pfaffian functions, which is more understood nowadays. So we introduce some tools which may be useful in this project.

Firstly we introduce change of variables.

\begin{definition} \rm
A map $\Phi:U\subseteq\mathbb{R}\sp{n}\to V\subseteq \mathbb{R}\sp{m}$ is a \emph{Pfaffian map} (respectively \emph{Noetherian map}) if each of its coordinate functions is a Pfaffian (respectively Noetherian) function.

Let $f_1,\dots,f_N:V\subseteq\mathbb{R}^m\to\mathbb{R}$ be a Noetherian or Pfaffian chain and $\Phi:U\subseteq\mathbb{R}\sp{n}\to V\subseteq \mathbb{R}\sp{m}$ a Noetherian or Pfaffian map. The sequence $g_i(\bar x)=f_i\circ\Phi(\bar x)$, $1\leq i\leq N$, is called the \emph{pull back} of the chain $f_1$, \dots, $f_N$ (by the map $\Phi$).
\end{definition}

\begin{proposition}
The pull-back of a Pfaffian or Noetherian chain by a Pfaffian or Noetherian map can be extended to a Noetherian chain.
\end{proposition}

\begin{proof}
This is a simple application of the chain rule in the calculation of derivatives of the composition of functions.

We name the variables $\bar{x}=(x_1,\dots,x_n)\in\mathbb{R}^n$ and $\bar{y}=(y_1,\dots,y_m)\in\mathbb{R}^m$. Write $\Phi(\bar x)=(\phi_1(\bar x),\dots,\phi_m(\bar x))$. Let $\psi_1(\bar{x})$, \dots, $\psi_M(\bar{x})$ be a Noetherian (or Pfaffian) chain containing the coordinate functions of $\Phi$.

Let $g_i(\bar x)=f_i\circ\Phi(\bar x)$, $1\leq i\leq N$, be the pull back of the Noetherian chain $f_1(\bar{y})$, \dots, $f_N(\bar{y})$. We apply the chain rule to calculate the differentials
$$
dg_i(\bar{x})=\sum_{j=1}^{n}\left(\sum_{j=1}^{m}\frac{\partial g_i}{\partial y_j}\circ(\Phi(\bar{x}))\,\frac{\partial\phi_j}{\partial x_i}\right)dx_i.
$$
The partial derivatives of the coordinate functions of $\Phi$ are polynomials on the variables $\bar{x}$ and the functions $\psi_k$, $1\leq k\leq M$. The partial derivatives of the functions $f_j(\bar{y})$ are polynomials in the variables $\bar{y}$ and the functions $f_i$. The composition with the coordinate functions of the map $\phi$ turn these into polynomias in the variable $\bar{x}$ and the functions $g_j$.

The sequence $\psi_1$, \dots, $\psi_M$, $g_1$, \dots, $g_M$ if the desired Noetherian chain.
\end{proof}

\medskip

Now we introduce modifications on the Noetherian chain in a lemma whose proof is similar to the previous one.

\begin{lemma}
Let $\Psi:W\subseteq\mathbb{R}^k\to\mathbb{R}^{k+p}$ be a Pfaffian map. Let $f_1,\dots,f_k:V\subseteq\mathbb{R}^m\to\mathbb{R}$ be a sequence of functions, such that the image of the map $F(\bar x)=(f_1(\bar x),\dots,f_k(\bar x))$ is contained in $W$. If the functions $g_1$, \dots, $g_{k+p}:V\to\mathbb{R}$ are such that $(g_1(\bar x)$, \dots, $g_{k+p}(\bar x))=\Psi\circ F(\bar x)$, and can be extended to a Pfaffian chain, then $f_1$, \dots, $f_k$ can be extended to a Noetherian chain.
\end{lemma}

\begin{example}[Trigonometric Functions] \rm
The sequence of two functions $g_1(x)=\cos x$, $g_2(x)=\sin x$ is a Noetherian chain, proving that the sine and cosine functions are Noetherian Functions.

The sine function is not a Pfaffian function on $\mathbb{R}$ because it has infinitely many zeros, but its restriction to the open interval $]-\pi/2,\pi/2[$ is Pfaffian. Consider the sequence $f_1(x)=\tan x$, $f_2(x)=\sec x=1/\cos x$, $f_3(x)=\cos x$ and $f_4(x)=\sin x$, all of them restricted to the interval  $]-\pi/2,\pi/2[$. This is a Pfaffian chain because
$$
\begin{array}{rclcl}
f_1'(x) & = & \sec^2 x & = & 1+f_1^2(x) \\
f_2'(x) & = & \sec x\cdot\tan x & = & f_1(x)\cdot f_2(x) \\
f_3'(x) & = & -\sin x & = & -f_1(x)\cdot f_3(x) \\
f_4'(x) & = & \cos x & = & f_3(x),
\end{array}
$$
where $f_j'$ denotes the derivative of $f_j$, $1\leq j\leq 4$. The interval  $]-\pi/2,\pi/2[$ appears naturally as a connected component of the domain of $\tan x$.

Notice that here we used $\Phi(x)=x$ and $\Psi(y_1,y_2)=({^{y_2}\!/_{\!y_1}},{^1\!/\!_{y_1}},\,y_1,\,y_2)$ applied to $(g_1(x),g_2(x))$.
\end{example}

This example is simple because the pair of functions $\sin x$ and $\cos x$ satisfies a first order linear system of two differential equations. For equations of higher order, we are not yet able to give a positive or negative answer to the problem. This is the subject of the following section.

We are now able to state our first problem.

\begin{problem}\label{open-problem-1}
Given a Noetherian chain, can it be locally extended to a Pfaffian chain, or at least to another Noetherian chain which is the pull-back of a Pfaffian chain under a Pfaffian map? Can this be done recursively?
\end{problem}

Lou van den Dries asked the following question about unrestricted Pfaffian functions (see, \cite{speissegger2012} for a discussion and partial results).

\begin{problem}\label{model-complete-pfaffian}
Is the expansion of the field of real numbers by unrestricted Pfaffian functions model complete?
\end{problem}

See the comments after the Open Problem \ref{pfaffian-hypergeometric-model-complete}, p. \pageref{pfaffian-hypergeometric-model-complete}, for connections with the definability of the exponential function. We know today that such expansion is o-minimal by \cite{wilkie1999}.

\section{Complex Linear Differential Equations}\label{complex-differential-equations}

We consider some special cases related to linear differential equations of first and second order. We show that solutions to complex first order linear equations the real and imaginary parts of the solutions are locally Pfaffian. For second order linear equations with non constant coefficients we run into difficulties because there appears a non linear first order equation (a Riccati equation) which is not amenable to the same treatment given to the linear case. It is worth mentioning the relation between quotients of two linearly independent solutions to a second order linear differential equation and modular functions (see \cite[Chapter XI]{ford1929}).

\subsection{First Order Equations}

Given holomorphic functions $g(z)$ and $h(z)$, we consider the complex first order linear differential equation $Y'=g(z)Y+h(z)$, where the prime indicates derivative with respect to the complex variable $z=x+iy$.

\begin{lemma}
 Suppose that $f(z)$ satisfies a first order linear equation $Y'=g(z)Y+h(z)$, such that the real and imaginary parts of the functions $g$ and $h$ are Pfaffian functions. Then the real and imaginary parts of $f$ are locally Pfaffian functions.
\end{lemma}

\begin{proof}
We proceed in two steps. We first solve the associated homogeneous equation $Y'=g(z)Y$. Write $g(z)=a(x,y)+ib(x,y)$ and let $f_0(z)=u_0(x,y)+iv_0(x,y)$ be a non zero solution. Then
$$
\frac{df_0}{dz}=\frac{1}{2}\left(\frac{\partial}{\partial x}-i\frac{\partial}{\partial y}\right)(u_0+iv_0)=
\frac{1}{2}\left(\frac{\partial u_0}{\partial x}+\frac{\partial{v_0}}{\partial y}\right)+\frac{i}{2}\left(\frac{\partial v_0}{\partial x}-\frac{\partial u_0}{\partial y}\right).
$$
We use the Cauchy-Riemman equations $\frac{\partial u_0}{\partial x}=\frac{\partial v_0}{\partial y}$ and $\frac{\partial u_0}{\partial y}=-\frac{\partial v_0}{\partial x}$ to obtain
$$
\frac{df_0}{dz}=\frac{\partial u_0}{\partial x}+i\frac{\partial v_0}{\partial x}=\frac{\partial v_0}{\partial y}-i\frac{\partial u_0}{\partial y}.
$$
The differential equation becomes the system
$$
\arraycolsep=1.4pt\def\arraystretch{2.2}
\begin{array}{rclrcl}
\displaystyle\frac{\partial u_0}{\partial x} & = & au_0-bv_0,  & ~~~ \displaystyle\frac{\partial u_0}{\partial y} & = & -bu_0-av_0  \\
\displaystyle\frac{\partial v_0}{\partial x} & = & bu_0+av_0,  & \displaystyle\frac{\partial v_0}{\partial y} & = & au_0-bv_0
\end{array}
$$

If we set $q_0=u_0/v_0$, $q_1=1/v_0$, $q_2=v_0$ and $q_3=u_0$ (and this may impose a restriction to the domain of definition of the functions), then
$$
\frac{\partial q_0}{\partial x}=\frac{1}{v_0}\frac{\partial u_0}{\partial x}-\frac{u_0}{v_0^2}\frac{\partial v_0}{\partial x}=-b(1+q_0^2),
$$
and going the same way as the example of the sine function, we obtain a chain $q_0$, $q_1$, $q_2$, $q_3$, satisfying
$$
\arraycolsep=1.4pt\def\arraystretch{2.2}
\begin{array}{rclrcl}
\dfrac{\partial q_0}{\partial x} & = & -b(1+q_0^2), & ~~\dfrac{\partial q_0}{\partial y} & = & -a(1+q_0^2)\\
\dfrac{\partial q_1}{\partial x} & = & -(bq_0+a)q_1, & \dfrac{\partial q_1}{\partial y} & = & -(aq_0-b)q_1 \\
\dfrac{\partial q_2}{\partial x} & = & (bq_0+a)q_2, & \dfrac{\partial q_2}{\partial y} & = & (aq_0-b)q_2 \\
\dfrac{\partial q_3}{\partial x} & = & -(2bq_0^2-aq_0+b)q_2, & \dfrac{\partial q_3}{\partial y} & = &-(bq_0-a)q_2,
\end{array}
$$
which implies that the chain is a Pfaffian chain.

\medskip

Now, suppose that $Y=f(z)=u(x,y)+iv(x,y)$ satisfies $Y'=g(z)Y+h(z)$. We use the {\em method of variation of parameters}, \cite[pp. 60--62]{birkhoff-rota1989}, to solve the equation. We use the function $f_0(z)$ and search for a solution $f(z)=c(z)f_0(z)$. Substitute $y=f(z)$ in the equation and we obtain $f_0(z)c'(z)=h(z)$, or $c'(z)=h(z)/f_0(z)$. Write $h(z)=h_1(x,y)+ih_2(x,y)$ and $h/f_0=(h_1+ih_2)/(u_0+iv_0)=[(h_1u_0+h_2v_0)+i(h_2u_0-h_1v_0)]/(u_0^2+v_0^2)$.
Since $h_1$ and $h_2$ are Pffafian functions (by hypothesis), and $u_0$ and $v_0$ are Pfaffian functions by the above, we apply Lemma \ref{lema-basic-properties} to conclude that the real and imaginary parts of $f(z)$ (with a possibly smaller domain) are Pfaffian functions.
\end{proof}

\medskip

As a particular example we consider the complex exponential function.

\begin{example} \rm
 The real and imaginary parts of the complex exponential function are (locally) Pfaffian functions. We write the function $f(z) = \exp z=\exp(x+iy)=u(x,y)+i\,v(x,y)$. Then the equation $f'(z)=f(z)$ satisfied by the exponential function can be written as the system
$$
\arraycolsep=1.4pt\def\arraystretch{2.2}
\begin{array}{rclrcl}
\dfrac{\partial u}{\partial x} & = & u; & ~~~ \dfrac{\partial u}{\partial y} & = & -v \\
\dfrac{\partial v}{\partial x} & = & v; & ~~~ \dfrac{\partial v}{\partial y} & = & u.
\end{array}
$$

The Pfaffian chain is closely related to that of the sine function: $g_1(x,y)=u/v$, $g_2(x,y)=1/v$, $g_3(x,y)=v$, $g_4(x,y)=u$.
\end{example}

\subsection{Second Order and Riccati Equations}

We consider here linear homogeneous second order linear equations
$$
Y''+a_1(z)Y'+a_0(z)Y=0,
$$
with complex meromorphic coefficients $a_1(z)$ and $a_0(z)$.

The substitution $q=Y'/Y$ transforms this equation into the \emph{Riccati's equation} (see \cite[pp. 45--46]{birkhoff-rota1989})
$$
q'+q^2+a_1(z)q+a_0(z)=0.
$$
If we write $q(z)=q(x+iy)=u(x,y)+iv(x,y)$ and $a_i(z)=A_i(x,y)+B_i(x,y)$ ($i=0,1$), then Riccati's equation becomes the system
$$
\arraycolsep=1.4pt\def\arraystretch{2.2}
\begin{array}{rclrcl}
\dfrac{\partial u}{\partial x} & = & u^2-v^2 +A_1u-B_1v+A_0, & ~~~\dfrac{\partial u}{\partial y} & = &  \dfrac{\partial v}{\partial x}\\
\dfrac{\partial v}{\partial x} & = & 2uv+B_1u+A_1v+B_0,  & \dfrac{\partial v}{\partial y} & = &  -\dfrac{\partial u}{\partial x}
\end{array}
$$

If $A_i$ and $B_i$ ($i=0,1$) are Noetherian functions, then $u$ and $v$ are Noetherian functions.

\begin{problem}\label{Riccati-Pfaffian}
Is there a map $\Psi$ which transforms the pair $(u,v)$ into a Pfaffian chain (or at least into a sequence contained in a Pfaffian chain)?
\end{problem}

If we restrict to the real functions we obtain a Pfaffian functions $q$, $Y$, satisfying $q'=q^2+a_1(x)q+a_0(x)$ and $Y'=qY$. If all these functions are real analytic in a neighbourhood of the interval $[x_0,x_1]$, $x_0<x_1$, then we obtain the following result by Wilkie's method, \cite{wilkie1996}.

\begin{theorem}
Let $f_1$, \dots, $f_N$ be a Pfaffian chain of functions defined in some open neighbourhood of the interval $[x_0,x_1]$ which contains the coefficients and a non zero solution of $q'=q^2+a_1(x)q+a_0(x)$ and $Y'=qY$. Let $g_1$, \dots, $g_N$ be defined as $g_j(x)=f_j(x)$ if $x_0\leq x\leq x_1$, and as zero elsewhere, $1\leq j\leq N$. Then the theory of the structure $\langle\mathbb{R},0,1,-,+,\cdot,g_1,\dots,g_N\rangle$ is model complete.
\end{theorem}

An important class of second order linear differential equations is the class of hypergeometric equations. We treat them in the following section.

\section{Hypergeometric Equations and Functions}\label{hypergeometric-functions}

In this section we treat the case of the hypergeometric differential equation. We first summarize the basics about hypergeometric differential equations and the Gauss hypergeometric functions (one of the solutions). Then we prove a model completeness result for expansions of the reals by suitable restrictions of the hypergeometric function and the unrestricted exponential function. The next subsection contains results about definability and model completeness for expansions of the reals by the real and imaginary parts of hypergeometric functions. Finally we deal with a particular case of ${_2F_1}({^1\!/_2},{^1\!/_2};1;z)$, which has a close relation to the modular functions.

\subsection{Preliminaries}

We summarize here some facts about the hypergeometric functions and their respective second order linear differential equations (see, for instance, \cite[Chapter 2]{erdelyi1953}, or \cite[Chapter XVI]{whittaker-watson}).

The \emph{hypergeometric differential equation} is the equation
$$
z(1-z)Y''+[c-(a+b+1)z]Y'-abY=0,
$$
where $a,b,c\in\mathbb{C}$, $(-c)\not\in\mathbb{N}$.

One solution is given for $|z|<1$ by Gauss's \emph{hypergeometric series}
$$
{_2F_1}(a,b;c;z)=\sum_{n=0}^{\infty}\frac{(a)_n(b)_n}{(c)_n\,n!}z^n,
$$
where $(x)_n$ is the Pochhammer symbol, defined as $(x)_0=1$ and $(x)_{n+1}=(x)_n(x+n)$, for all $n\geq 0$. It is easy to see that
$$
\frac{d}{dz}\,{_2F_1}(a,b;c;z)=\frac{\Gamma(a)\Gamma(b)}{\Gamma(c)}\,{_2F_1}(a+1,b+1;c+1;z),
$$
where $\Gamma(z)$ is the Gamma Function. Except the cases where $a$ or $b$ is a non positive integer where ${_2F_1}(a,b;c;z)$ is a polynomial, the hypergeometric functions have branching points at $z=1$ and $z=\infty$. They are single valued in the complex plane minus the real interval $[1,\infty)$.

\emph{Euler's Formula} \cite[\S\,2.1.3 (10)]{erdelyi1953} allows us to define analytic continuations of the hypergeometric function. It is the integral
$$
F(a,b;c;z)=\frac{\Gamma(c)}{\Gamma(b)\Gamma(c-b)}\int_0^1t^{b-1}(1-t)^{c-b-1}(1-zt)^{-a}\,dt,
$$
for $|\arg(1-z)|<\pi$, $\Re(c)>\Re(b)>0$.

We apply these results to definability and model completeness results in what follows.

\subsection{The Real Case}

We treat firstly the case of one real variable hypergeometric functions because they can be formalized in the Pfaffian function setting.

Here we restrict the parameters $a,b,c$ to the real numbers. Euler's formula in this case holds for $z$ in the real interval $(-\infty,1)$, for $c>b>0$. In that interval the integrand is positive and so the (real) hypergeometric functions do not assume the value zero. The function $q(x)=F'(a,b;c;x)/F(a,b;c;x)=F(a+1,b+1;c+1;x)/F(a,b;c;x)$ is defined for all $x\in(-\infty,1)$.

If we restrict the functions $f(x)=F(a,b;c;x)$ and $q(x)=\frac{F'(a,b;c;x)}{F(a,b;c;x)}$ to any interval $[x_0,x_1]$, with $x_0<x_1<1$, and defined as zero elsewhere, then Wilkie's method \cite{wilkie1996} gives the following result.

\begin{theorem}
The first order theory of $\langle \mathbb{R},0,1,+,\cdot,-,f,q\rangle$ is model complete.
\end{theorem}

The method of Section \ref{model-completeness-test} above is more appropriate to the complex case, studied below.

If we restrict the functions $f$ and $q$ to the unbounded interval $(-\infty,1)$, and define as zero elsewhere, we may have a logarithmic singularity at $x=1$. Therefore the best we can prove at the moment is the following result, based on Wilkie's \cite{wilkie1999}.

\begin{theorem}
 The structure $R_H=\langle \mathbb{R},0,1,+,\cdot,-,f,q\rangle$ is o-minimal.
\end{theorem}

At the present knowledge, we can only state the following question.

\begin{problem}\label{pfaffian-hypergeometric-model-complete}
Is the theory of $R_H$ model-complete?
\end{problem}

A positive answer to this problem would imply that the unrestricted exponential function is existentially definable in such structure (it is definable, by \cite{miller1994}). Notice though that $f(x)=f(x_0)\exp\left(\int_{x_0}^xq(t)\,dt\right)$.

\subsection{The Complex Case}

The hypergeometric series ${_2F_1}(a,b;c;z)$ define an analytic function if $|z|<1$ which can be analytically continued to the complex plane minus the real interval $[1,\infty)$, in general with branching points at $z=1$ and $z=\infty$. It defines a solution to the hypergeometric differential equation
$$
z(1-z)Y''+[c-(a+b+1)z]Y'-abY=0.
$$

If the parameter $c\not\in\mathbb{Z}$, then a second solution is $y_2(z)=z\sp{1-c}{_2F_1}(1+a-c,1+b-c;2-c;z)$, defined in $\mathbb{C}\setminus\left((-\infty,0]\cup(-\infty,0]\right)$, with branching point at $z=0$.

Let $F_0(z)={_2F_1}(a,b;c;z)$ if $z\in\mathbb{C}\setminus[1,\infty)$, $F_1(z)=F_0'(z)$ (the derivative), $F_2(z)={_2F_1}(1+a-c,1+b-c;2-c;z)$, $z\in\mathbb{C}\setminus[1,\infty)$, $F_3(z)=F_2'(z)$, and $F_j(z)=0$ elsewhere, $j=0,1,2,3$.

We firstly prove an auxiliary result.

\begin{lemma}[Monodromy]\label{monodromy-lemma-F}
The main branch of the functions $F_0(z)$ and $F_2(z)$ (and their derivatives. $F_1$ and $F_3$), are defined in the domain $\mathbb{C}\setminus[1,\infty)$, where we choose $\arg(1-z)\in(-\pi,\pi)$. The analytic continuations to $-3\pi<\arg(z)<-\pi$ and $\pi<\arg(z)<3\pi$ are given by linear combinations of $F_0$ and $F_2$, (respectively, of $F_1$ and $F_3$).
\end{lemma}

\begin{proof}
This is a direct application of \cite[\S\ 2.7.1, Formulas 1--3, p. 93]{erdelyi1953}. 
\end{proof}

\begin{theorem}
 The main branch and adjacent branches of $F_j(z)$, $0\leq j\leq 3$, are definable in $\mathbb{R}_{\mathit{an},\exp}$.
\end{theorem}

\begin{proof}
The Monodromy Lemma \ref{monodromy-lemma-F} implies the definability of the adjacent branches of $F_j(z)$, $0\leq j\leq 3$, once we prove the definability of their main branches.

\smallskip

The function $\arctan x$, $x\in\mathbb{R}$ is definable in $\mathbb{R}_{\mathit{an},\exp}$, because of the formula
$$
\arctan(1/x)=\frac{\pi}{2}-\arctan(x)=2\arctan(1)-\arctan(x),~~x\neq 0,
$$
so it is definable from its restriction to the interval $[-1,1]$. This function together with the real exponential function allow us to define any particular branch of the complex logarithm. Therefore, we can define the power function $z^{1-c}$, for any $c\in\mathbb{C}$.

\noindent\dotfill

We only prove that $F_0(z)$ is definable in $\mathbb{R}_{\mathit{an},\exp}$. The proof for the other functions are analogous. Formulas for analytic continuation allow us to define other branches of $F_0$ (see \cite[\S\ 2.10, Formulas (1)-(6), pp. 108-109]{erdelyi1953}).

\smallskip

The function $z\mapsto z^2$ maps the region $\{z\in\mathbb{C}:\Re(z)>0\}=\{z\in\mathbb{C}:|\arg(z)|<\pi/2\}$ onto the region $\{z\in\mathbb{C}:|\arg(z)|<\pi\}$, so the function $z\mapsto [1-(1-z)^2]=(2z-z^2)$ maps $\{z\in\mathbb{C}:|\arg(1-z)|<\pi/2\}$ onto $\{z\in\mathbb{C}:|\arg(1-z)|<\pi\}$.
The M\"obius transformation $z\mapsto[2z/(z+1)]$ maps the unit disk $D_{1}(0)$ onto $\{z\in\mathbb{C}:|\arg(1-z)|<\pi/2\}$; maps the arc $e^{i\theta}$, $-\pi<\theta<0$, onto the ray $1-iy$, $y>0$; and maps the arc $e^{i\theta}$, $0<\theta\leq 0$ onto the ray $1+iy$, $y>0$. It maps the interval $[1,\infty)$ onto the interval $[1,2)$; the interval $(-\infty,-1]$ onto $[2,\infty)$, and $(-1,1)$ onto $(-\infty,1)$.
Therefore the composition of these functions $z\mapsto w=[2z/(z+1)]\mapsto(2w-w^2)=[4z/(z+1)^2]$ maps the disk $D_{1}(0)$ onto the region $\{z\in\mathbb{C}:|\arg(1-z)|<\pi\}$.

The function $G(z)=F_0(4z/(z+1)^2)$ is analytic in the unit disk $D_1(0)$ with branching points at $z=1$ and $z=-1$. These can be removed using the formula \cite[\S\ 2.9, Formula (35), p. 107]{erdelyi1953}.

This gives the desired definability result.
\end{proof}

Now let $RF_j(x,y)=\Re(F_j(x+iy))$, $IF_j(x,y)=\Im(F_j(x+iy))$, $0\leq j\leq 3$.

\begin{theorem}
 The first order theory of the structure $R_H=\langle\mathbb{R},\mathit{constants},-$, $+,\cdot,RF_0$, $RF_1$, $RF_2$, $RF_3$, $IF_0$, $IF_1$, $IF_2$, $IF_3,\exp,\arctan\rangle$ is o-minimal and strongly model complete.
\end{theorem}

\begin{proof}
The o-minimality is a consequence of the previous theorem and the o-minimality of $\mathbb{R}_{\mathit{an},\exp}$.

Model completeness follows from the method of Section \ref{model-completeness-test}.

The transformation $z\mapsto 4z/(z+1)^2$ maps the unit disk to the region $\mathbb{C}\setminus[1,\infty)$, so $G_j(z)=F_j(z)$ is defined in the unit disk, if $H_{0,j}(z)$ and $H_{1,j}(z)$ are the two adjacent branches of $F_j(z)$, then $H_{i,j}(4z/(1+z)^2)$ are defined in the region outside the disk (one in each halfplane $\mathbb{H}$ and $-\mathbb{H}$). This gives the definition of the analytic continuation of $G_j(z)$ to the disk $D_2(0)$, required by Theorem \ref{theorem-model-completeness-test}. The derivatives of all orders of these functions can be defined from the corresponding second order hypergeometric differential equation.
\end{proof}

\subsection{The Function $\boldsymbol{{_2F_1}(\sfrac{1}{2},\sfrac{1}{2};1;z)}$}

In this section we consider the \emph{complete elliptic integral of the first kind}, which is a particular case of hypergeometric functions,
$$
K(z)=\int_0^1\frac{1}{\sqrt{1-t^2}\sqrt{1-zt^2}}\,dt=\int_0^{\pi/2}\frac{d\theta}{\sqrt{1-z\sin^2\theta}}=\frac{\pi}{2}\cdot{_2F_1}(\sfrac{1}{2},\sfrac{1}{2};1;z).
$$

This is one of the solutions to the equation
$$
(z-z^2)Y''+(1-2z)Y'-\frac{Y}{4}=0,
$$
which has the functions $y_1(z)=K(z)= (\pi/2)\,{_2F_1}(\sfrac{1}{2},\sfrac{1}{2};1;z)$ and $y_2(z)=iK(1-z)$ as a pair of $\mathbb{C}$-linearly independent solutions.

We can view $K(z)$ as a multivalued complex analytic function in the variable $z\in\mathbb{C}\setminus\{1\}$, or a single valued function with the variable $z$ restricted to the domain $\mathbb{C}\setminus[1,\infty)$ (the complex numbers minus the set of real numbers greater or equal to 1). The point $z=1$ is a logarithmic branch point for the integral and so for each $z\neq 1$ there are infinitely many possible values for $K(z)$ (one for each \textit{branch} fo $K(z)$).

We choose the main branch of $K(z)$, for $z\in \mathbb{C}\setminus[1,\infty)$, by choosing the positive square roots in the integrand when $0<z<1$ and taking their analytic continuations. We intend to prove a model completeness result for expansions of the field of real numbers by the real and imaginary parts of $K(z)$. In order to do this we should be able to define analytic continuations to other branches of $K(z)$.

We do this in two steps. Firstly we define an extension of $K(z)$ for real $z>1$ and then we extend to the other branches using the \textit{monodromy matrices}.

Recall that an \emph{argument of a complex number} $w\in\mathbb{C}$, $w\neq 0$ is some $\theta\in\mathbb{R}$, denoted $\arg(w)$, such that $w=|w|\,e^{i\,\theta}$.

\begin{lemma}[Monodromy]\label{monodromy-lemma}
The main branch of the function $K(z)$ and its derivative $K'(z)$, are defined in the domain $\mathbb{C}\setminus[1,\infty)$, where we choose $\arg(1-z)\in(-\pi,\pi)$. These can be continued to the real interval $[1,\infty)$ and to adjacent branches ($-3\pi<\arg(1-z)<\pi$, and $\pi<\arg(1-z)<3\pi$)  by linear combinations of $K(z)$ and $K(1-z)$, and of $K'(z)$ and $K'(z)$, respectively.
\end{lemma}

\begin{proof}
We prove the result for the analytic continuation of $K(z)$ to the branch $-3\pi<\arg(1-z)<\pi$, and to the interval
$[1,\infty)$ with $\arg(1-z)=-\pi$. The other cases are analogous.

Write $z=k^2$, for $k>1$. We write the integral as the sum of integrals in the intervals $0\leq t\leq 1/k$, where $\sqrt{1-k^2t}$ is real, and $1/k<t\leq 1$, where the square root is pure imaginary, and here we make an explicit choice $\sqrt{1-k^2t}=-i\sqrt{k^2t-1}$. Therefore
$$
K(k^2)=\int_0\sp{1/k}\frac{1}{\sqrt{1-t^2}\sqrt{1-k^2t^2}}\,dt+i\int_{1/k}^1\frac{1}{\sqrt{1-t^2}\sqrt{k^2t^2-1}}\,dt,
$$
where we choose the branch $\sqrt{w}>0$, for $w>0$.

With this choice, the function $K(z)$ becomes discontinuous in the set $[1,\infty)$, but continuous from below, that is, for $x>1$,
$$
\lim_{y\to 0\sp{+}}K(x-iy)=K(x),~~\lim_{y\to 0\sp{+}}K(x+iy)=-2iK(1-x)+K(x),
$$
(see \cite[Section 2.7.1, Formulas (4) and (5), p. 95]{erdelyi1953}).

From this, we obtain the monodromy matrix
$$
M=\left(
\begin{array}{rr}
1 & -2i \\
0 & 1
\end{array}
\right)
$$

This finishes the proof.
\end{proof}

\begin{theorem}
The main branch of the function $K(z)$ and its derivative $K'(z)$, $z\in\mathbb{C}\setminus[1,\infty)$, and their analytic continuation to the adjacent branches are definable in $\mathbb{R}_{\mathit{an},\mathrm{exp}}$.
\end{theorem}

\begin{proof}
We prove the result only for $K(z)=(\pi/2){_2F_1}(\sfrac{1}{2},\sfrac{1}{2};1;z)$ because for its derivative $K'(z)=(\pi/2){_2F_1}(\sfrac{3}{2},\sfrac{3}{2};2;z)$ the argument is analogous.

We use change of variables to transform the function into an analytic function defined in an open neighbourhood of a closed disk of finite radius. The main branch complex function $\log(1-z)$ is definable in $\mathbb{R}_{\mathit{an},\exp}$, and we use it to remove the logarithmic singularities at $z=1$ and $z=\infty$.

\medskip

We recall that the function $z\mapsto z^2$ maps the region $\{z\in\mathbb{C}:\Re(z)>0\}=\{z\in\mathbb{C}:|\arg(z)|<\pi/2\}$ onto the region $\{z\in\mathbb{C}:|\arg(z)|<\pi\}$, so the function $z\mapsto [1-(1-z)^2]=(2z-z^2)$ maps $\{z\in\mathbb{C}:|\arg(1-z)|<\pi/2\}$ onto $\{z\in\mathbb{C}:|\arg(1-z)|<\pi\}$.
The M\"obius transformation $z\mapsto[2z/(z+1)]$ maps the unit disk $D_{1}(0)$ onto $\{z\in\mathbb{C}:|\arg(1-z)|<\pi/2\}$; maps the arc $e^{i\theta}$, $-\pi<\theta<0$ onto the ray $1-iy$, $y>0$; and maps the arc $e^{i\theta}$, $0<\theta\leq 0$ onto the ray $1+iy$, $y>0$. It maps the interval $[1,\infty)$ onto the interval $[1,2)$; the interval $(-\infty,-1]$ onto $[2,\infty)$, and $(-1,1)$ onto $(-\infty,1)$.
Therefore the composition of these functions $z\mapsto w=[2z/(z+1)]\mapsto(2w-w^2)=[4z/(z+1)^2]$ maps the disk $D_{1}(0)$ onto the region $\{z\in\mathbb{C}:|\arg(1-z)|<\pi\}$.

The function $F(z)=K(4z/(z+1)^2)$ is analytic in the disk $D_{1}(0)$  and has logarithmic singularities at $z=1$ and $z=-1$. This is the composition of $G(w)=K(2w-w^2)$ with $w=2z/(z+1)$. The function $G(w)$ admits analytic continuation $\tilde{G}(w)$ to $\mathbb{C}\setminus[1,\infty)$ (by Lemma \ref{monodromy-lemma}). Therefore, $F(z)$ admits (definable) analytic continuation to $\mathbb{C}\setminus\left[ (-\infty,-1]\cup [1,\infty)\right]$.

\medskip

The following equation from \cite[\S\,2.7.1, Formula (6), p. 95]{erdelyi1953} is the key to eliminate the logarithmic singularities
$$
\frac{\pi}{2}\cdot{_2F_1}(\sfrac{1}{2},\sfrac{1}{2};1;z)+\frac{\log(1-z)}{2}\cdot{_2F_1}(\sfrac{1}{2},\sfrac{1}{2};1;1-z)={}
$$
$$
{}=\sum_{n=0}^{\infty}\left[\frac{({^1\!/_2})_n}{n!}\right]^2[\psi(n+1)-\psi(n+{^1\!/_2})](1-z)^n, \eqno{(\ast)}
$$
where $\psi(z)$ is the derivative of Euler's Gamma Function $\Gamma(z)$. Notice that its left hand side equals
$$
K(z)+\frac{\log(1-z)}{\pi}K(1-z).
$$

\smallskip

The formula $(\ast)$ above says that $K(z)+\frac{\log(1-z)}{\pi}K(1-z)$ is analytic in a disk around $z=1$. So we can isolate this branch point.

The second branch point lies at $\infty$ and we isolate it making the M\"obius transformation $z\mapsto z/(z-1)$, which maps the real interval $[1,\infty)$ onto itself and maps the point $\infty$ to $z=1$.

$G(z)=K((1-z)^2)$ is defined in the half plane $\Re(z)<1$. $H(z)=G(1/z)$ is defined in the open disk $D_{\sfrac{1}{2}}(\sfrac{1}{2})=\{z\in\mathbb{C}:{}$ $|z-\sfrac{1}{2}|<\sfrac{1}{2}\}$. It can be analytically continued around all points of its boundary, except the points $z=0$ and $z=1$ (which are logarithmic singularities).

The transformation $z\mapsto z/(z-1)$ maps $1\mapsto\infty$, $\infty\mapsto 1$, fixes $z=0$, maps the real interval $[0,1)$ onto the interval $(-\infty,0]$ (and \textit{vice-versa}), and keep invariant the real interval $[1,\infty)$. Then $K(z/(z-1))$ is also defined in the same domain as $K(z)$. Euler's integral produces
$$
K\left(\frac{z}{z-1}\right)=\frac{\pi}{2}\sqrt{1-z}\int_0^1\frac{dt}{\sqrt{t}\sqrt{1-t}\sqrt{1-z(1-t)}}={}
$$
$$
{}=\frac{\pi}{2}\sqrt{1-z}\int_0^1\frac{dt}{\sqrt{t}\sqrt{1-t}\sqrt{1-zt}}=\sqrt{1-z}K(z),
$$
where the last but one equality comes from the change of variables $t\mapsto (1-t)$ (this is called the Pfaff transformation). We must choose the branches of the square roots $\sqrt{1-z}$ and $\sqrt{1-z(1-t)}$ to provide the correct branch of $\sqrt{1-zt}$, that is, if $z\in(0,1)$, $\sqrt{1-z}>0$ and $\sqrt{1-z(1-t)}>0$.

\medskip

The function $z\mapsto(2z/(z+1)^2)$ maps the disk $D_{\rho}((\eta+1)/2\eta)$, where $\rho=(\eta^2-1)/2\eta$ ($0<\eta<1$), onto the disk $D_{\delta}(1)$, where $\delta=1-4\eta/(\eta+1)^2$. If $f(z)$ is an analytic function defined on $D_{\rho^2}(1)$, then $g(z)=f(2z/(z+1)^2)$ is analytic in $D_{1}(1)\supseteq \overline{D_{\delta}(1)}$ (the closure of $D_{\delta}(1)$). We take $f_1(z)=K(z)+\log(1-z)\,K(z)/\pi$, which is analytic in the disk $D_1(1)$ and its restriction to $\overline{D_{\delta}(1)}$ is definable in $\mathbb{R}_{\mathit{an},\exp}$.

Therefore $K(z)$ restricted to $\mathbb{C}\setminus[1,\infty)$ (and defined as zero in $[1,\infty)$ is definable in $\mathbb{R}_{\mathit{an},\exp}$.
\end{proof}

This allows us to prove the following model completeness result.

\medskip

Let $RK_k(x,y)$ and $IK_k(x,y)$, $k=0,1$, be the real and imaginary parts of the functions $K(x+iy))$ and its derivative $K'(x+iy)$, respectively, for $x+iy\not\in [1,\infty)$, and defined as zero in that interval.

\begin{theorem}\label{model-complete-K}
The theory of the structure $R_K=\langle\mathbb{R}$, constants, $+,\cdot,-$, $\exp$, $\arctan$, $RF_0,RF_1$, $IF_0,IF_1\rangle$ is strongly model complete, where $\exp$ and $\arctan$ are unrestricted exponential and arctangent functions.
\end{theorem}

\begin{proof}
We deduce from the proof of the previous theorem together with Theorem \ref{theorem-model-completeness-test}, page \pageref{theorem-model-completeness-test}, the model completeness of the expansion $\langle\mathbb{R}$, constants, $+,\cdot,-$, $\exp,\arctan,RG_0,RG_1$, $IG_0,IG_1\rangle$, where $RG_j(x,y)=\Re[K\sp{(j)}(4z/(z+1)^2)]$, $IG_j(x,y)=\Im[K\sp{(j)}(4z/(z+1)^2)]$, $j=0,1$, $|z|<1$ and defined as zero if $|z|\geq 1$.
This gives the definition of the analytic continuation of $G_j(z)$ to the disk $D_2(0)$, required by Theorem \ref{theorem-model-completeness-test}. The derivatives of all orders of these functions can be defined from the corresponding second order hypergeometric differential equation.

Each of the structures admit strong existential interpretations in the other, so both are strongly model complete.
\end{proof}

Because $\tau(z)=iK(1-z)/K(z)$ is the inverse of the modular function $z=\lambda(\tau)$ restricted to the set $\mathcal{F}=\{\tau\in\mathbb{H}: |\Re(\tau)|\leq 1;\, |2\tau\pm 1|\geq 1\}$ (see \cite[\S\ 4.4, pp. 76--84]{venkatachaliengar2012},  or \cite[\S\ 2.7.4, p. 99]{erdelyi1953}), we have the following corollary.

\begin{corollary}
The modular functions $\lambda, j:\mathcal{F}\to\mathbb{C}$ is strongly definable in the structure $R_K=\langle\mathbb{R},\mathit{constants},+,\cdot,-,\exp,\arctan,RF_0,RF_1$, $IF_0,IF_1\rangle$.
\end{corollary}

\begin{remark}\rm
 There is an algebraic relation between $j(z)$ and $\lambda(z)$, namely,
 $$
 j(z)=\frac{256(1-\lambda+\lambda^2}{\lambda^2(1-\lambda)^2},
 $$
 which allows us to strongly define this function in the structure $R_K$. (See \cite[Chapter VII, \S\S\ 8-9, pp. 116-118]{chandrasekhar1985}. See also \cite{bianconi2015} for some model completeness results related to the modular $j$ function.)
\end{remark}

\subsection{Restriction to the Real Numbers}

If we restrict $K(z)$ and its derivative $K'(z)$  to the real interval $(-\infty,1)$, then they are Pfaffian functions, and therefore the expansion of the field of the real numbers by $K(z)$, and its derivative $K'(z)$, is o-minimal by \cite{wilkie1999}.

\begin{problem}\label{exp-definable-2F1-real}
Is the exponential function existentially definable from the restriction of the function ${_2F_1}(\sfrac{1}{2},\sfrac{1}{2};1;z)$ and its derivative ${_2F_1}'(\sfrac{1}{2},\sfrac{1}{2};1;z)$ (this is equal to ${_2F_1}(\sfrac{3}{2},\sfrac{3}{2};2;z)$) to the interval $(-1,1)$?
\end{problem}

Because the singularity at $x=1$ is logarithmic, the exponential function is certainly definable by \cite[Theorem, pp. 257-258]{miller1994}.

\section{On Wilkie's Conjecture}\label{wilkie-conjecture}

In this section we discuss Wilkie's conjecture, a counting statement about rational points in definable sets. In this context model completeness results may prove to be useful because definable sets are existentially definable, and so of low quantifier complexity.

Jonathan Pila and Alex Wilkie proved in \cite{pila-wilkie2006} a counting result about rational points in definable sets in the structure $\mathbb{R}\sb{\mathit{an},\exp}$ which proved to be useful in applications to algebraic geometrical problems, \cite{pila2011,klinger-ullmo-yafaev2014,gao2015}. In that paper it was conjectured a sharper counting result for the real exponential field, which could be true in some other particular cases (it is not true in the whole $\mathbb{R}_{\mathit{an},\exp}$). A first positive result is Binyamini and Novikov's \cite{binyamini-novikov2016}, where they prove it for the expansion of the real field by restricted sine and exponential functions. They use the strong model completeness techniques, which are not recursively computable. The restricted sine and exponential functions are Pfaffian functions, so they raise the question whether their proof can be done with the techniques of Macintyre and Wilkie, \cite{macintyre-wilkie1996}, which could give computable bounds to the counting arguments.

We recall the following definitions.

\begin{definition}[Algebraic Part of a Set]
Let $A\subseteq\mathbb{R}^n$ be a non empty set. The \emph{algebraic part} $A\sp{\mathit{alg}}\subseteq A$ of $A$ is the union of all connected semialgebraic subsets of $A$. The \emph{transcendental part} of $A$ is the set $A\sp{\mathit{trans}}=A\setminus A\sp{\mathit{alg}}$.
\end{definition}

\begin{definition}[Height of a Rational Number]
Let $r=a/b\in\mathbb{Q}$, with either $r=a=0$ or $\gcd(a,b)=1$. The \emph{height} of $r$ is the number $\max\{|a|,|b|\}$.
\end{definition}

Wilkie and Pila have proved for a set $A$ definable in $\mathbb{R}_{\mathit{an},\exp}$ that the number of points of $A\sp{\mathit{trans}}$ with rational coordinates with height at most $H$ is $\mathcal{O}(H^{\alpha})$, for some $\alpha>0$.

\smallskip

\textbf{Wilkie's Conjecture} is the bound $\mathcal{O}((\log H)^{\alpha})$ for sets definable in $\mathbb{R}_{\exp}$.

\begin{problem}\label{wilkie-conjecture-pfaffian}
Can we prove Wilkie's conjecture for expansions of the real field by a Pfaffian chain restricted to compact poly-intervals? If so, which cases can be done recursively?
\end{problem}

In another direction we have the following problem.

\begin{problem}\label{wilkie-conjecture-modular}
Is Wilkie's conjecture true for expansions of the real field by elliptic and modular functions?
\end{problem}

A positive answer to this problem would imply the original conjecture conjecture for the real exponential field because Peterzil and Starchenko have proved in \cite[Theorem 5.7, p. 545] {peterzil-starchenko-wp2004} that the real exponential function is definable from the $\wp$ function (actually, from the modular function $j(z)$, which is itself definable from $\wp(z)$).

\section{Final Remarks}\label{final-remarks}

The themes of Pfaffian functions and model completeness permeate this paper. It seems to be far from exhausted and the problems posed in this paper touch a few aspects of this subject.

Because we have focussed in the case of complex differential equations and hypergeometric functions, we have not touched in one of the major problems dealing with Pffafian functions. Alex Wilkie proved in 1991 (published in 1996, \cite{wilkie1996}) the model completeness of expansions of the real field by restricted Pfaffian functions, and in 1999 he proved that the expansion of the real field by unrestricted Pfaffian functions is o-minimal (see \cite{wilkie1999}). It remains to prove (or disprove) the model completeness of expansions by unrestricted Pfaffian functions (the positive answer is known today as \textit{van den Dries Conjecture}; see discussion in \cite{speissegger2012}). But this is another story.

\medskip

\textsc{Acknowledgements.} The author wishes to thank the organizers of the Workshop ``{\em Logic and Applications: in honour to Francisco Miraglia by the occasion of his 70th birthday}'', Professors Marcelo Esteban Coniglio, Hugo Luiz Mariano and Vinicius Cif\'u Lopes for the invitation to give a 50 minutes talk, and Professor Marcelo Esteban Coniglio by the encouragement to publish a paper on the subject of the talk in the \emph{South American Journal of Logic}. This forced us to organize the ideas to put it in a more or less definite form, which helped to finish the proofs of some of the conjectured results presented at the event. The author also wishes to thank Professor Francisco Miraglia, our old friend \emph{Chico}, which always inspired and encouraged us with his enthusiasm in our mathematical careers.



\noindent Department of Mathematics\\
Instituto de Matem\'atica e Estat\'istica (IME-USP)\\
University of S\~ao Paulo (USP)\\
Rua do Mat\~ao, 1010, Butant\~a, CEP 05508-090, S\~ao Paulo, SP,  Brazil\\ 
email: bianconi@ime.usp.br

\end{document}